\documentclass[a4paper, oneside, 11pt]{amsart}

\usepackage{algorithm}
\usepackage[noend]{algorithmic}
\usepackage{amsfonts}
\usepackage{amsmath}
\usepackage{amssymb}
\usepackage{amsthm}
\usepackage[english]{babel}
\usepackage{hyperref}

\theoremstyle{plain}
\newtheorem{theorem}{Theorem}[section]
\newtheorem{proposition}[theorem]{Proposition}
\newtheorem{corollary}[theorem]{Corollary}
\theoremstyle{definition}
\newtheorem{definition}[theorem]{Definition}

\DeclareMathOperator{\lcm}{LCM}

\newcommand{\Q}{\ensuremath{\mathbb{Q}}}
\newcommand{\Singular}{\textsc{Singular}}

\begin{document}

\title[On the Primary Decomposition of Some Determ. Hyperedge Ideal]%
{On the Primary Decomposition \\ of Some Determinantal Hyperedge Ideal}

\author{Gerhard Pfister}
\address{Gerhard Pfister\\
Department of Mathematics\\
TU Kaiserslautern\\
Erwin-Schr\"odinger-Str.\\
67663 Kaiserslautern\\
Germany}
\email{pfister@mathematik.uni-kl.de}

\author{Andreas Steenpa\ss}
\address{Andreas Steenpa\ss\\
Department of Mathematics\\
TU Kaiserslautern\\
Er\-win-Schr\"odinger-Str.\\
67663 Kaiserslautern\\
Germany}
\email{steenpass@mathematik.uni-kl.de}


\date{\today}

\begin{abstract}
In this paper we describe the method which we applied to successfully compute
the primary decomposition of a certain ideal coming from applications in
combinatorial algebra and algebraic statistics regarding conditional
independence statements with hidden variables. While our method is based on the
algorithm for primary decomposition by Gianni, Trager and Zacharias, we were
not able to decompose the ideal using the standard form of that algorithm, nor
by any other method known to us.
\end{abstract}

\maketitle

\section{Introduction}

The aim of this paper is to describe the method which we applied to compute the
primary decomposition of a certain ideal. The
importance of this ideal comes from applications in combinatorial algebra and
algebraic statistics, in particular regarding conditional independence
statements, see \cite{M}, \cite{MR}, \cite{EH} or \cite{EHHM}.

Conditional independence (CI) is an important tool in statistical modelling.
There has been a lot of recent activity on ideals associated to conditional
independence statements, when all random variables are observed. However, in
applications such as causal reasoning, some variables are hidden and it is
important to know what constraints on the observed variables are caused by
hidden variables.

From algebraic point of view, knowing the primary decomposition of CI ideals
leads to important results in causal inference. We have studied a
particular ideal of interest in causal inference whose primary decomposition
has led to interesting theoretical and computational results. It is shown in
this paper that this ideal is equidimensiomal and has two prime components. In
particular, one of the prime components is not a determinantal ideal which is a
surprising result. In addition, both components are important in the sense of
causal inference, as none of them contains a variable, see
\cite[Example~4.2]{CMR}.

We tried the algorithms for primary decomposition which are available in
\Singular{}, that is, the one of Gianni, Trager and Zacharias \cite{GTZ} as
well as the one of Shimoyama and Yokoyama \cite{SY}, but none of
them succeeded to decompose the ideal within one month.

The outline of this paper is quite simple: The input ideal and some of its
combinatorial properties are given in Section~\ref{sec:input_ideal}. This ideal
can be decomposed using the computational method presented in
Section~\ref{sec:computational_method}. This method combines four improvements
to the well-known algorithm for primary decomposition by Gianni, Trager and
Zacharias \cite{GTZ} which are described in detail in separate subsections.
Finally, in Section~\ref{sec:result}, we specify the result of the primary
decomposition and reveal some of the combinatorial properties of the computed
prime components. However, the combinatorial structure of the second prime
component is not yet fully understood.

\section{Input ideal}%
\label{sec:input_ideal}

The ideal $I$ of which we would like to compute a primary decomposition is a
determinantal hyperedge ideal. Over the polynomial ring
$R := \Q[x_1, \ldots, x_{12},\, y_1, \ldots, y_{12},\, z_1, \ldots, z_{12}]$,
consider the $3 {\times} 12$-matrix
\[
M :=
\begin{pmatrix}
x_1 & x_2 & x_3 & x_4 & x_5 & x_6 & x_7 & x_8 & x_9 & x_{10} & x_{11} & x_{12}
\\
y_1 & y_2 & y_3 & y_4 & y_5 & y_6 & y_7 & y_8 & y_9 & y_{10} & y_{11} & y_{12}
\\
z_1 & z_2 & z_3 & z_4 & z_5 & z_6 & z_7 & z_8 & z_9 & z_{10} & z_{11} & z_{12}
\end{pmatrix} .
\]
Following the notation for determinantal hyperedge ideals in \cite{CMR}, set
\begin{align*}
R_1 &:= \{ 1, 2, 3 \} =: N, & C_1 &:= \{ 1, 4, 7, 10 \}, \\
R_2 &:= \{ 4, 5, 6 \},      & C_2 &:= \{ 2, 5, 8, 11 \}, \\
R_3 &:= \{ 7, 8, 9 \},      & C_3 &:= \{ 3, 6, 9, 12 \}, \\
R_4 &:= \{ 10, 11, 12 \},
\end{align*}
and denote the minor of $M$ with row indices $A$ and column indices $B$ by
$[A|B]_M$. Furthermore, for any set $X$ and any non-negative integer $s$, we
denote by $\binom{X}{s}$ the set of all subsets of $X$ which contain exactly
$s$ elements.

With this notation, we define the ideal $I \subseteq R$ as
\[
I :=
\left\langle [N|B]_M \;\middle|\;
B = R_i,\, i \in \{ 1, 2, 3, 4 \} \;\text{ or }\;
B \in \binom{C_j}{3},\, j \in \{ 1, 2, 3 \} \right\rangle,
\]
that is,
{
\allowdisplaybreaks
\begin{align*}
I = \langle
& x_1 y_4 z_7 - x_1 y_7 z_4 - x_4 y_1 z_7 + x_4 y_7 z_1 + x_7 y_1 z_4
  - x_7 y_4 z_1, \\
& x_1 y_4 z_{10} - x_1 y_{10} z_4 - x_4 y_1 z_{10} + x_4 y_{10} z_1
  + x_{10} y_1 z_4 - x_{10} y_4 z_1, \\
& x_1 y_7 z_{10} - x_1 y_{10} z_7 - x_7 y_1 z_{10} + x_7 y_{10} z_1
  + x_{10} y_1 z_7 - x_{10} y_7 z_1, \\
& x_4 y_7 z_{10} - x_4 y_{10} z_7 - x_7 y_4 z_{10} + x_7 y_{10} z_4
  + x_{10} y_4 z_7 - x_{10} y_7 z_4, \\
& x_2 y_5 z_8 - x_2 y_8 z_5 - x_5 y_2 z_8 + x_5 y_8 z_2 + x_8 y_2 z_5
  - x_8 y_5 z_2, \\
& x_2 y_5 z_{11} - x_2 y_{11} z_5 - x_5 y_2 z_{11} + x_5 y_{11} z_2
  + x_{11} y_2 z_5 - x_{11} y_5 z_2, \\
& x_2 y_8 z_{11} - x_2 y_{11} z_8 - x_8 y_2 z_{11} + x_8 y_{11} z_2
  + x_{11} y_2 z_8 - x_{11} y_8 z_2, \\
& x_5 y_8 z_{11} - x_5 y_{11} z_8 - x_8 y_5 z_{11} + x_8 y_{11} z_5
  + x_{11} y_5 z_8 - x_{11} y_8 z_5, \\
& x_3 y_6 z_9 - x_3 y_9 z_6 - x_6 y_3 z_9 + x_6 y_9 z_3 + x_9 y_3 z_6
  - x_9 y_6 z_3, \\
& x_3 y_6 z_{12} - x_3 y_{12} z_6 - x_6 y_3 z_{12} + x_6 y_{12} z_3
  + x_{12} y_3 z_6 - x_{12} y_6 z_3, \\
& x_3 y_9 z_{12} - x_3 y_{12} z_9 - x_9 y_3 z_{12} + x_9 y_{12} z_3
  + x_{12} y_3 z_9 - x_{12} y_9 z_3, \\
& x_6 y_9 z_{12} - x_6 y_{12} z_9 - x_9 y_6 z_{12} + x_9 y_{12} z_6
  + x_{12} y_6 z_9 - x_{12} y_9 z_6, \\
& x_1 y_2 z_3 - x_1 y_3 z_2 - x_2 y_1 z_3 + x_2 y_3 z_1 + x_3 y_1 z_2
  - x_3 y_2 z_1, \\
& x_4 y_5 z_6 - x_4 y_6 z_5 - x_5 y_4 z_6 + x_5 y_6 z_4 + x_6 y_4 z_5
  - x_6 y_5 z_4, \\
& x_7 y_8 z_9 - x_7 y_9 z_8 - x_8 y_7 z_9 + x_8 y_9 z_7 + x_9 y_7 z_8
  - x_9 y_8 z_7, \\
& x_{10} y_{11} z_{12} - x_{10} y_{12} z_{11} - x_{11} y_{10} z_{12}
  + x_{11} y_{12} z_{10} + x_{12} y_{10} z_{11} - x_{12} y_{11} z_{10}
  \rangle \\
{} \subseteq {}
& \Q[x_1, \ldots, x_{12},\, y_1, \ldots, y_{12},\, z_1, \ldots, z_{12}]
\,.
\end{align*}
}

The ideal $I$ exhibits a kind of symmetry in the sense that it stays
invariant under certain permutations of the variables in the polynomial ring
$R$. In fact, the following operations on the matrix $M$ in the above
construction lead to the same ideal $I$:
\begin{itemize}
\item
permuting the rows
\item
permuting the sets of columns which correspond to the index sets $R_1$, $R_2$,
$R_3$ and $R_4$
\item
permuting the sets of columns which correspond to the index sets $C_1$, $C_2$
and $C_3$
\end{itemize}

Let $S_n$ be the symmetric group of order $n$ as usual. Then the above
observations imply that the symmetry group $\mathcal{S} \leq S_{36}$ of $I$ has
a subgroup $\mathcal{S'}$ isomorphic to $S_3 \times S_4 \times S_3$. We believe
that indeed $\mathcal{S} \cong \mathcal{S'}$. However, this is more difficult
to prove.

\section{Computational method}%
\label{sec:computational_method}

Our computational method is based on the GTZ algorithm for primary
decomposition by Gianni, Trager and Zacharias \cite{GTZ} as presented in
\cite[Chapter 4]{GP}. Let us recall the fact that the GTZ algorithm computes
some maximal dimensional associated primary ideals $Q_1, \ldots, Q_s$ of $I$
using maximal independent sets to reduce the problem to the zero-dimensional
case, as well as a polynomial $h^m \notin I$ such that
$I : \langle h^m \rangle = I : \langle h^\infty \rangle
= Q_1 \cap \ldots \cap Q_s$
and $I = \left( I : \langle h^m \rangle \right) \cap \langle I, h^m \rangle$.

An implementation of this algorithm can be found in the \Singular{} library
\texttt{primdec.lib} \cite{DLPS}. However, neither this implementation nor the
one of the algorithm of Shimoyama and Yokoyama \cite{SY} in the same library
succeeded to compute a primary decomposition of the input ideal $I$ defined in
the previous section within one month.

We propose four improvements for applying this algorithm to our particular
example:
\begin{enumerate}
\item
carefully choosing the maximal independent sets, based on data from
intermediate steps;
\item
if an ideal $P$ is expected to be a prime ideal, applying a special algorithm
to check the primality of $P$;
\item
choosing a monomial order for saturation which is compatible with the monomial
orders used in the previous steps;
\item
making use of the symmetry of the input ideal for the saturation step.
\end{enumerate}

These four improvements will be explained in detail in the following
subsections. The last improvement cannot easily be applied in general if the
input ideal does not admit any useful symmetry. On the other hand, we expect
that especially the first and the third improvement are also beneficial for
other classes of ideals.

\subsection{Choosing maximal independent sets}

In the GTZ algorithm, the reduction to the zero-dimensional case requires a
maximal independent set, see \cite[Algorithm~4.3.2]{GP}. However, the input
ideal might admit several maximal independent sets, any one of which can be
used in the subsequent steps.

For the convenience of the reader, we recall the definition of maximal
independent sets:
\begin{definition}[{\cite[Definition~3.5.3]{GP}}]
Let $I \subset K[x_1, \ldots, x_n]$ be an ideal. Then a subset
\[
u \subset x = \{ x_1, \ldots, x_n \}
\]
is called an independent set (with respect to $I$) if $I \cap K[u] = 0$.
An independent set $u \subset x$ (with respect to $I$) is called maximal if
$\dim(K[x]/I) = |u|$.
\end{definition}

In our case, the \Singular{} command \verb+indepSet(., 0)+, when applied to a
Gr\"obner basis of $I$ w.r.t.\@ the degree reverse lexicographic order with
$x_1 > \ldots > x_{12} > y_1 > \ldots > y_{12} > z_1 > \ldots > z_{12}$, yields
$17,223$ different maximal independent sets. The ideal $I$ may even have more
maximal independent sets, but it is computationally hard to list all of them.

Depending on the choice made in the algorithm, the running time needed for the
computation of the next primary component, especially for the saturation step,
may vary enormously. Thus we propose to proceed as follows: Let $K[X]$ with
$X = \{ X_1, \ldots, X_n \}$ be the polynomial ring of the input ideal.
Randomly choose a maximal independent set $u \subseteq X$ of $I$ and compute
a minimal Gr\"obner basis $G_u \subseteq K[X]$ of the zero-dimensional ideal
$I K(u)[X \setminus u] \subseteq K(u)[X \setminus u]$, see
\cite[Proposition~4.3.1~(1)]{GP}, where $K(u)$ is the field of fractions of
$K[u]$. Further compute
\[
d_u
:= {\dim_{K(u)[X \setminus u]}}
{\left( K(u)[X \setminus u] \bigm/ I K(u)[X \setminus u] \right)}
\]
as well as the degrees and the number of terms of the leading coefficients of
the elements in $G_u$, considered as polynomials in $K[u]$. Repeat this for
several maximal independent sets $u$ and sort them, in this order,
\begin{enumerate}
\item
by increasing vector space dimension $d_u$,
\item
by increasing (maximal) degree and finally
\item
by the (maximal) number of terms of the leading coefficients of the elements in
$G_u$.
\end{enumerate}
This defines a partial order on the maximal independent sets. Choose one of the
first sets w.r.t.\@ this order and proceed with the GTZ
algorithm as usual.

It is reasonable to expect, and in fact our experiments confirm, that on
average the above choice leads to improved running times for the saturation
step compared to other, for example random, choices.

There is a tradeoff between the number of maximal independent sets for which
the above mentioned data is computed and the overall speedup which can be
achieved by this method. In our example, we used about 700 randomly chosen
maximal independent sets. Finally we subsequently picked
\begin{align*}
u_1 &:= \{ x_1, \ldots, x_{12}, y_1, \ldots, y_{12}, z_1, z_2 \} \;
\text{ and then} \\
u_2 &:= \{ x_1, x_2, x_4, x_5, x_8, x_9, x_{11}, x_{12}, y_1, \ldots, y_{12},
z_1, z_4, z_8, z_9, z_{11}, z_{12} \}
\end{align*}
to find the two primary components described in
Section~\ref{sec:result}. Especially the computation of the second primary
component was far out of reach for many other choices of maximal independent
sets.

\subsection{Checking ideals for primality}

While trying to compute a primary decomposition of some given ideal $I$ by
whatever means, one may face the situation that one of the computed ideals $P$
is expected to be one of the prime components of $I$, but it is not known a
priori whether or not $P$ is indeed a prime ideal. For example, this may happen
if $I$ is defined in a polynomial ring over a field of characteristic zero, but
$P$ has been found via computations in positive characteristic. In that
situation, Algorithm~\ref{alg:primality_check} can be used to check the
primality of $P$. This algorithm may be known to some experts in the field, but
we did not find any reference.

\begin{algorithm}
\caption{Primality check}%
\label{alg:primality_check}
\begin{algorithmic}[1]
\REQUIRE $P \subseteq K[X] = K[X_1, \ldots, X_n]$
\ENSURE true if $P$ is a prime ideal, false otherwise
\STATE choose a maximal independent set $u \subseteq X$ of $P$
\STATE compute a minimal Gr\"obner basis $G \subseteq K[X]$ of
$P K(u)[X \setminus u]$
\IF {$P K(u)[X \setminus u]$ is not a maximal ideal in $K(u)[X \setminus u]$}
\RETURN false
\ENDIF
\STATE let $c_1, \ldots c_s \in K[u]$ be the leading coefficients of the
elements in $G$
\FOR {$i = 1, \ldots, s$}
\IF {$P \neq P : \langle c_i^\infty \rangle$}%
\label{line:sat}
\RETURN false
\ENDIF
\ENDFOR
\RETURN true
\end{algorithmic}
\end{algorithm}

The correctness of this algorithm is based on the following statement:

\begin{proposition}%
\label{prop:primality_check}
Let $P \subseteq K[X] = K[X_1, \ldots, X_n]$ be an ideal and let
$u \subseteq X$ be a maximal independent set of $P$. Then $P$ is a prime ideal
if and only if $P K(u)[X \setminus u]$ is a maximal ideal in
$K(u)[X \setminus u]$ and $P = P K(u)[X \setminus u] \cap K[X]$.
\end{proposition}

\begin{proof}
The ideal $P K(u)[X \setminus u]$ has dimension zero by
\cite[Proposition~4.3.1]{GP}.
Let $S = K[u] \setminus \{0\}$, then $K(u)[X \setminus u]$ is the localization
of $K[X]$ at $S$. Therefore the result follows from
\cite[Proposition~3.11(iv)]{AM}.
\end{proof}

\begin{corollary}
The output of Algorithm~\ref{alg:primality_check} is correct.
\end{corollary}

\begin{proof}
Let $G \subseteq K[X]$ be a minimal Gr\"obner basis of $P K(u)[X \setminus u]$,
let $c_1, \ldots c_s \in K[u]$ be the leading coefficients of the elements in
$G$ as in the algorithm, and define $h := \lcm(c_1, \ldots c_s) \in K[u]$. Then
\[
P K(u)[X \setminus u] \cap K[X]
= P : \langle h^\infty \rangle
= \left( \ldots \left( P : \langle c_1^\infty \rangle \right) \ldots
: \langle c_s^\infty \rangle \right)
\]
by \cite[Proposition~4.3.1~(2)]{GP} and the basic properties of ideal
quotients. Thus the statement follows from
Proposition~\ref{prop:primality_check}.
\end{proof}

In our case, we made use of the above algorithm to confirm the primality of the
component $P_2$, see Section~\ref{sec:result}. Computationally, this was by far
the hardest part of finding a primary decomposition of the input ideal $I$
defined in Section~\ref{sec:input_ideal}.

A variant of Algorithm~\ref{alg:primality_check} is the following improvement
to the GTZ algorithm: Let $I = \langle f_1, \ldots, f_m \rangle \subseteq K[X]$
be the remaining ideal to decompose, of dimension greater than zero, and let
$u \subseteq X$ be a maximal independent set of $I$. Then as the next step, the
algorithm computes a primary decomposition
$I K(u)[X \setminus u] = Q_1 \cap \ldots \cap Q_s$ in dimension zero. As the
algorithm proceeds, the intersections $Q_i \cap K[X]$ are computed via
saturation of $Q_i$ w.r.t.\@ some $d_i \in K[u]$ to obtain primary components
of $I$, see \cite[paragraph below Algorithm~4.3.4]{GP}.

Now suppose $s = 1$, that is, $I K(u)[X \setminus u]$ is a primary ideal in
$K(u)[X \setminus u]$. Let $\{ g_1, \ldots, g_r \} \subseteq I \subseteq K[X]$
be a Gr\"obner basis of $I K(u)[X \setminus u] = Q_1$. Such a Gr\"obner basis
always exists: For example, take a lexicographical Gr\"obner basis of $I$ with
$X \setminus u > u$ and discard the elements not needed for a minimal Gr\"obner
basis of $Q_1$. The next step in the GTZ algorithm would be the saturation of
$I_0 := \langle g_1, \ldots, g_r \rangle_{K[X]} \subseteq I$ w.r.t.\@
$c \in K[u]$, the least commom multiple of the leading coefficients of the
$g_i$ considered as elements of $K(u)[X \setminus u]$.
But in this case, we have
$I_0 : \langle c^\infty \rangle = I : \langle c^\infty \rangle$.
Usually, it is computationally easier to saturate $I$ than $I_0$ because
loosely speaking, $I$ is already closer to the final result
$I_0 : \langle c^\infty \rangle$ than $I_0$.

In the case where $Q_1$ is even a maximal ideal, this trick is computationally
equivalent to Algorithm~\ref{alg:primality_check}.

\subsection{Choosing a monomial order for saturation}%
\label{ssec:order}

For ideals of dimension greater than zero, the GTZ algorithm requires a
saturation step, see \cite[Chapter~4.3]{GP}. As for many applications of
Gr\"obner bases, the running time of this step heavily depends on the choice of
the monomial order. Our experiments have shown that it is beneficial to choose
a monomial order which is compatible with the maximal independent set that has
been used in the preceding steps of the algorithm.

More precisely, let $K[X]$ with $X = \{ X_1, \ldots, X_n \}$ be the polynomial
ring of the input ideal as in the previous subsection and let $u \subseteq X$
be the chosen maximal independent set. We propose to use an elimination order
for $X \setminus u$ on $K[X]$ for the saturation step. For implementational
reasons, the lexicographical order with $X \setminus u > u$ is a good choice,
see \cite[Algorithm~4.3.2, step~2]{GP}.

\subsection{Making use of symmetry for the saturation step}%
\label{ssec:symmetry}

With notation as in Algorithm~\ref{alg:primality_check}, if the input ideal $P$
has some kind of symmetry which also occurs among the coefficients $c_i$, then
this can be used to speed up the computation. More precisely, let $\varphi$ be
a $K$-algebra automorphism of $K[X]$ which leaves $P$ invariant, that is,
$\varphi(P) = P$. Furthermore suppose that $\varphi(c_j) = c_i$ for some
indices $i, j \in \{ 1, \ldots, s \}$ with $i \neq j$, and that
$P = P : \langle c_j^\infty \rangle$ has already been checked in
line~\ref{line:sat} of Algorithm~\ref{alg:primality_check}. Because applying
$\varphi$ commutes with saturation, we then have
\[
P : \langle c_i^\infty \rangle
= \varphi(P) : \langle \varphi(c_j)^\infty \rangle
= \varphi(P : \langle c_j^\infty \rangle)
= \varphi(P) = P \,.
\]
Thus the check can be left out for $c_i$.

For computing a primary decomposition of the input ideal $I$ defined in
Section~\ref{sec:input_ideal}, it turned out that $\varphi(c_j) = c_i$ for some
indices $i, j \in \{ 1, \ldots, s \}$ with $i \neq j$ where $\varphi$ is the
$\Q$-algebra automorphism of
\[
\Q[x_1, \ldots, x_{12},\, y_1, \ldots, y_{12},\, z_1, \ldots, z_{12}]
\]
given by
\begin{align*}
\varphi(x_i) &:= x_{\sigma(i)}, \\
\varphi(y_i) &:= y_{\sigma(i)}, \\
\varphi(z_i) &:= z_{\sigma(i)},
\end{align*}
with $\sigma = (1 4) (2 5) (3 6) \in S_{12}$. The
automorphism $\varphi$ leaves $I$ and the two prime components $P_1$ and $P_2$
presented in Section~\ref{sec:result} invariant, which is easy to see for $I$
and $P_1$ and can be checked computationally for $P_2$. We used
Algorithm~\ref{alg:primality_check} to check the primality of $P_2$, so in our
case $P = P_2$. Thus computing the saturation
$P : \langle c_i^\infty \rangle$ is superfluous. However, we have performed
this computation once and it took almost 500~hours which is more than four
times as long as the saturation of $P$
w.r.t.\@ $c_j$, see Section~\ref{sec:result}. Since in this case, $\varphi$ is
just a permutation of variables, this example underlines once more the
importance of choosing a monomial order for the saturation step as discussed in
Subsection~\ref{ssec:order}.

\section{Main Result}%
\label{sec:result}

Using the computational method from the previous section, we were able to show
the following:
\begin{theorem}
The ideal
$I \subseteq R
= \Q[x_1, \ldots, x_{12},\, y_1, \ldots, y_{12},\, z_1, \ldots, z_{12}]$
defined in Section~\ref{sec:input_ideal} is the intersection
\[
I = P_1 \cap P_2
\]
of two different prime ideals $P_1, P_2$ where all three ideals $I$, $P_1$ and
$P_2$ have Krull dimension $26$. In particular, this implies that the above
intersection is the unique irredundant primary decomposition of $I$.
\end{theorem}

In our case, we first found the prime component $P_1$, with
$I \subsetneqq P_1$, which was relatively easy. We then found the ideal $P_2$
for which we were able to show that $I = P_1 \cap P_2$ and $I \subsetneqq P_2$.
But the hardest part of the computation were the saturations for the primality
check of $P_2$, see line~\ref{line:sat} in Algorithm~\ref{alg:primality_check}.
While checking that $P_2$ is already saturated w.r.t.\@ all the coefficients
$c_i$ was easy in some cases, the longest of these computations took 117~hours
on an Intel\textsuperscript{\tiny\textregistered} Core\texttrademark{} i7-6700
CPU with up to 4~GHz. For one coefficient, it takes almost
500~hours, but this computation can be avoided by the method described in
Subsection~\ref{ssec:symmetry}.

$P_1$ is the ideal generated by \emph{all} $3 {\times} 3$-minors
of the matrix $M$ defined in Section~\ref{sec:input_ideal}. The
structure of the prime component $P_2$ is more involved. Using the \Singular{}
command \verb+mstd()+, one can find a generating set $G$ of $P_2$ which has the
following properties\footnote{The set $G$ can be found in \Singular{}-readable
form at the very end of the \TeX{} file of this article which can be downloaded
from \url{https://arxiv.org/abs/1811.09530}.}:

\begin{itemize}
\item
$G$ consists of 44 polynomials $p_1, \ldots, p_{44}$. The elements
$p_1, \ldots, p_{16}$ are just the generators of $I$ listed in
Section~\ref{sec:input_ideal}.
\item
Each polynomial in $G' := \{ p_{17}, \ldots, p_{44} \}$ is homogeneous of
degree~12.
\item
The coefficients in $p_1, \ldots, p_{44}$ are all $1$, $-1$, $2$ or $-2$ where
the vast majority is just $1$ or $-1$.
\item
In each monomial in $G'$, each index $i$ from $1$ to $12$ appears exactly once,
as $x_i$, $y_i$ or $z_i$.
\item
For each single element in $G'$, all monomials have the same number of $x$'s,
$y$'s and $z$'s. For example, all monomials of $g_{17}$ have exactly six
$y$'s and six $z$'s, but no $x$'s. We denote this partition by $(0, 6, 6)$.
\item
The 28 elements of $G'$ are in one-to-one correspondence to the 28 possible
partitions of twelve items into three categories with no more then six items
belonging to the same category.
\item
For $g_{17}$ which corresponds to the partition $(0, 6, 6)$ as mentioned above,
we get all monomials as follows: From the $3 {\times} 4$-matrix
\[
\begin{pmatrix}
1 & 4 & 7 & 10 \\
2 & 5 & 8 & 11 \\
3 & 6 & 9 & 12
\end{pmatrix}
\]
choose two different columns A, B. Choose
two indices from A. Choose two indices from B such that they are not both in
the same rows as the two indices chosen from A. Let C, D be the remaining
columns. Choose one index from C. Choose one index from D which is in a
different row then the one chosen in C. The chosen indices are the $y_i$'s, all
the others are the $z_i$'s (or vice versa, of course). Repeating this process
for all possible choices, we get exactly the 216 monomials appearing in
$g_{17}$.
\item
The generators corresponding to the partitions $(6, 6, 0)$ and $(6, 0, 6)$ are
just the same as $g_{17}$ for some permutation of $(x, y, z)$.
\item
The generator $g_{19}$ corresponds to the partition $(1, 5, 6)$. It can be
mapped to $g_{17}$ via the map $x_i \mapsto y_i$. Some of the 252 terms of
$g_{19}$ cancel under this map.
\end{itemize}

The above properties indicate that the set $G$ has a combinatorial structure
which would be interesting to understand completely, including the
coefficients. So far, however, we were not able to reveal every detail of this
structure.

\section{Acknowledgements}

We would like to thank numerous anonymous referees for their feedback which
helped us to improve this article, as well as Fatemeh Mohammadi for clarifying
the importance of the discussed ideal and its primary decomposition in
algebraic statistics.

\end{document}